\documentclass[english]{amsart}

\def\nb{{\Bbb N}}

\def\rb{{\Bbb R}}
\def\cb{{\Bbb C}}
\def\zb{{\Bbb Z}}

 \usepackage{amsmath}

\theoremstyle{plain}
\newtheorem{theorem}{Theorem}[section]
\newtheorem{proposition}[theorem]{Proposition}

\newtheorem{lemma}[theorem]{Lemma}

\theoremstyle{definition}

\catcode`\é=\active
\defé{\'e}

\catcode`\è=\active
\defè{\`e}

\catcode`\ê=\active
\defê{\^e}

\catcode`\û=\active
\defû{\^u}

\catcode`\â=\active
\defâ{\^a}

\catcode`\à=\active
\defà{\`a}

\catcode`\ù=\active
\defù{\`u}

\catcode`\ç=\active
\defç{\c c}

\catcode`\î=\active
\defî{\^{\i}}

\catcode`\ï=\active
\defï{\"{\i}}

\catcode`\Ü=\active
\defï{\"U}

\catcode`\ô=\active
\defô{\^o}

\catcode`\é=\active
\defé{\'e}

\begin{document}
\title[Sums over zeros of functions from the Selberg class]{New asymptotic formulas for sums over zeros of
functions from the Selberg class}

\author[K.~Mazhouda]{Kamel Mazhouda}
\address{K. Mazhouda, Département de mathématiques, Faculté des sciences de Monastir, Monastir 5000, Tunisie}
\email{mazhoudakamel@yahoo.fr; kamel.mazhouda@fsm.rnu.tn}

\keywords{Selberg class,  sums over zeros, Weil explicit formulas}

\subjclass[2010]{ 11M06, 11M26, 11M41}

\begin{abstract}
In this paper, new asymptotic formulas for sums over zeros of functions from the Selberg
class are obtained. These results continue the investigations of Murty $\&$ Perelli \cite{12}, of
Murty $\&$ Zaharescu \cite{13}, of Kamiya $\&$ Suzuki \cite{8}, of Steuding \cite{19} and other authors.\\

\noindent R\'esum\'e. Dans ce article, nouvelles formules asymptotiques pour des sommes sur les zéros de fonctions de la classe de Selberg sont obtenus. Ces résultats continuent les recherches de Murty $\&$ Perelli \cite{12}, de Murty $\&$ Zaharescu \cite{13}, de Kamiya $\&$ Suzuki \cite{8}, de Steuding \cite{19} et d'autres auteurs.
\end{abstract}
\maketitle
%


\section{Introduction}\label{s:1}

We refer to the survey of Kaczorowski and Perelli \cite{6} and Selberg \cite{18}, for the definition and notations of the Selberg class.
The Selberg class $S$ consists of function $F(s)$ of a complex variable $s$ satisfying the following properties:\\
$i)$ (Dirichlet series) For $\Re(s)>1$,
$$F(s)=\sum_{n=1}^{+\infty}a_{F}(n)n^{-s},$$
where $a(1)=1$.\\
$ii)$ (Analytic continuation) For some integer $m\geq0$,\quad  $(s-1)^{m}F(s)$ \quad extends to an entire function of finite order.  We denote by \ $m_{F}$\  the smallest integer \ $m$\  which satisfies this condition .\\
$iii)$ (Functional equation) There are numbers $Q>0,\ \lambda_{j}>0$ and $\mu_{j}\in{\cb}$ with $\Re(\mu_{j})\geq0$, so that
$$\phi(s)=Q^{s}\prod_{j=1}^{r}\Gamma(\lambda_{j}s+\mu_{j})F(s)$$
satisfies
$$\phi(s)=\omega\overline{\phi}(1-s),$$
where $\omega$ is a complex number with $|\omega|=1$ and $\overline{\phi}(s)=\overline{\phi(\overline{s})}$.\\
$iv)$ (Euler product)
$$F(s)=\prod_{p}F_{p}(s),$$
where
$$F_{p}(s)=exp\left(\sum_{n=1}^{+\infty}\frac{b(p^{k})}{p^{ks}}\right),$$
and $b(p^{k})=O(p^{k\theta})$ for some $\theta<\frac{1}{2}$ and $p$ denotes a prime number.\\
$v)$ ( Ramanujan hypothesis ) For any fixed  $\epsilon>0$, \ $a(n)= O(n^{\epsilon})$.\\

\noindent It is expected that for every function in the Selberg class the analogue of the Riemann hypothesis holds, i.e, that all non trivial  zeros lie on the critical line $\Re(s)=\frac{1}{2}$. The degree of $F(s)\in{S}$ is defined by
$$d_{F}=2\sum_{j=1}^{r}\lambda_{j}.$$ The logarithmic
derivative of $F(s)$ also has the Dirichlet series expression
$$-\frac{F'}{F}(s)=\sum_{n=1}^{+\infty}\Lambda_{F}(n)n^{-s},$$
where $\Lambda_{F}(n)=b(n)\log n$ is an analogue of the Von Mongoldt
function $\Lambda(n)$ defined by
$$\Lambda(n)=\left\{\begin{array}{crll}\log p\quad  \hbox{if}&  n=p^{k}\  \hbox{with}\
k\geq1\\
0 & \hbox{otherwise.}\end{array}\right.$$
Let
 $$q_{F}=\frac{(2\pi)^{d_{F}}Q^{2}}{\beta}, \ \ \hbox{where}\ \ \beta=\prod_{j=1}^{r}\lambda_{j}^{-2\lambda_ {j}},$$
 be the conductor (or modulus) of $F\in{S}$.\\

  Let $\rho=\beta+i\gamma$ be  a  non trivial zero of some function in the Selberg class.
The aim of this  paper is to give an asymptotic formula for  sums involving zeros of  functions in the Selberg class $S$, that is,
$$\sum_{\rho}e^{u\rho^{2}-v\rho}, \hbox{ where}\  u>0\  \hbox{and}\  v\in{\rb}.$$
Furthermore,  we will discuss a more general quantity
$$\sum_{\rho}e^{v(\rho-\frac{1}{2})}\int_{-\infty}^{+\infty}f\left(\frac{x}{u}\right)e^{x(\rho-\frac{1}{2})}dx,$$
for fixed $v$ and  $u\rightarrow0^{+}$, where $f$ belongs to a certain class and show asymptotic results for it. For this purpose,  we use the Weil explicit formulas.   \\

\noindent The first main result of this paper is stated in the following Theorem.
\begin{theorem}\label{thm:1}
i) For $v=u$ or $v=0$, we have
$$
\sum_{\rho}e^{u\rho^{2}-v\rho}=-\frac{d_{F}}{4\sqrt{\pi u}}\log u+O_{F}(u^{-1/2}), \ \ \hbox{if}\ \ u\rightarrow0^{+},
$$
where  $ \rho$ runs over all non-trivial zeros  of $F(s)$ counted with multiplicity.\\
ii)   For any integer $m\geq2$, we have
$$\sum_{\rho}e^{u\rho^{2}+(\log m)\rho}=-\frac{\Lambda_{F}(m)}{\sqrt{4\pi u}}+O_{F,m}(1),\ \ \hbox{if}\ \ u\rightarrow0^{+},$$
and
$$\sum_{\rho}e^{u\rho^{2}-(\log m)\rho}=-\frac{\overline{\Lambda_{F}(m)}}{m\sqrt{4\pi u}}+O_{F,m}(1),\ \ \hbox{if}\ \ u\rightarrow0^{+}.$$
\end{theorem}
In this direction, assuming the Generalized Riemann ypothesis
(GRH), Murty and Perelli \cite{12} proved that, if $F\in{S}$,
then for $T>1$ and $n\in{\nb^{*}}$
 \begin{equation}
\sum_{|\gamma|\leq T}n^{\rho}=-\frac{T}{\pi}\Lambda_{F}(n)+O\left(n^{3/2}\log T\right),
\end{equation}
which is an extension to the Selberg class $S$ of the uniform version of Landau's formula obtained by Gonek \cite{3}.
Murty and Zaharescu \cite{13} proved unconditionally that, if\\
$F\in{S},\ x\geq2,\ \epsilon>0,\ n\in{\nb}$ and $n\geq x$ , then
\begin{eqnarray}
\sum_{|\gamma|\leq
T}x^{\rho}&=&-\frac{\Lambda_{F}(n)}{\pi}\frac{\sin(T\log\frac{T}{n})}{\log\frac{x}{n}}\nonumber\\&&+O_{\epsilon,F}(x^{1+\epsilon}\log^{2}T)+\ O\left(n^{1+\theta}\sum_{|n-p^{k}|<n^{\theta};p<p(\epsilon,F)}\frac{1}{|n-p^{k}|}\right),\nonumber
\end{eqnarray}
where $\rho$ runs over all non-trivial zeros of $F(s)$ and \ $p(\epsilon,F)$ depend uniquely on $F$ \\and $\epsilon$.\\

Actually, we will prove that i) of Theorem \ref{thm:1} is as follow
$$\sum_{\rho}e^{u\rho^{2}-v\rho}=\frac{d_{F}}{\sqrt{16\pi u}}\left(\log\frac{1}{u}-\gamma_{0}\right)+\frac{1}{\sqrt{4\pi u}}\log\left(\frac{q_{F}}{(4\pi)^{d_{F}}}\right)+O_{F}(1), \ \ \hbox{if}\ \ u\rightarrow0^{+},$$
where $\gamma_{0}$ is the Euler constant. Then, in the case of the classical Riemann zeta function, we get
$$\sum_{\rho}e^{u\rho^{2}-v\rho}=\frac{1}{\sqrt{16\pi u}}\log\frac{1}{u}-\frac{\log(16\pi^{2})+\gamma_{0}}{\sqrt{16\pi u}}+O(1),\  \hbox{if}\ \ u\rightarrow0^{+},$$
which was established  by  Kamiya and Suzuki in \cite{8}.\\

The main tool which will be used in the proofs of the second main result (Theorem \ref{thm:2} bellow) is the Weil explicit formulas given by the following Proposition.
 \begin{proposition}{\bf (Omar-Mazhouda \cite{14} \cite{15}) }\label{prop:1}
 Let $f$ be some complex-valued function on $\rb$ satisfying the conditions :\\
 a) $f$ is normalized, that is,
 $$f(x)=\frac{f(x^{+})+f(x^{-})}{2},\ \ x\in{\rb},$$
 where $f(x^{+})$ ( resp. $f(x^{-})$) means the right (resp. left) limit of $f$.\\
 b) There exist a constant $b>0$ such that
 $$V_{\rb}\left(f(x)e^{(\frac{1}{2}+b)|x|}\right)<\infty,$$
 where $V_{\rb}(.)$ means the total variation on $\rb$.\\
 c) There is a constant $\epsilon>0$ such that
 $$f(x)=\left\{\begin{array}{crll}f(0^{+})+O(|x|^{\epsilon}),&x\mapsto0^{+}\\
 f(0^{-})+O(|x|^{\epsilon}),&x\mapsto0^{-}.\end{array}\right.$$
 Then
 \begin{eqnarray} \sum_{\rho}\int_{-\infty}^{+\infty}&f(x)&e^{x(\rho-\frac{1}{2})}dx\nonumber\\
 &=&-\sum_{n=2}^{+\infty}\frac{\Lambda_{F}(n)}{\sqrt{n}}f(\log n)-\sum_{n=2}^{+\infty}\frac{\overline{\Lambda_{F}(n)}}{\sqrt{n}}f(-\log n)\\
 &&+\ m_{F}\left(\int_{-\infty}^{+\infty}f(x)e^{\frac{x}{2}}dx+\int_{-\infty}^{+\infty}f(x)e^{-\frac{x}{2}}dx\right)+2f(0)\log Q\nonumber\\
 &&+\ \sum_{j=1}^{r}\lambda_{j}\left(\frac{\Gamma'}{\Gamma}(\frac{\lambda_{j}}{2}+\mu_{j})+\frac{\Gamma'}{\Gamma}(\frac{\lambda_{j}}{2}+\overline{\mu_{j}})\right)f(0)\nonumber\\
&&-\ \sum_{j=1}^{r}\lambda_{j}\int_{0}^{+\infty}\left( f(-\lambda_{j}x)-f(0^{-})\right)\frac{e^{-(\frac{\lambda_{j}}{2}+\mu_{j})x}}{1-e^{-x}}dx\nonumber\\
&&-\ \sum_{j=1}^{r}\lambda_{j}\int_{0}^{+\infty}\left( f(\lambda_{j}x)-f(0^{+})\right)\frac{e^{-(\frac{\lambda_{j}}{2}+\overline{\mu_{j}})x}}{1-e^{-x}}dx.\nonumber
 \end{eqnarray}
\end{proposition}
\noindent This is the so called Weil explicit formulas with Mestre's formulation  \cite{11}. Proposition \ref{prop:1} is proved by a way similar to the proof of  the Weil explicit formulas  for function $f$ differ slightly in \cite{11} . There is no essential difference or difficulty in our case because of conditions
 for $F(s)$. Hence we omit the proof of Proposition \ref{prop:1}. \\

\noindent Replacing $f(x)$ in (2) by $f(\frac{x-v}{u})$, with the
assumption
$$f\left(\frac{x-v}{u}\right)=\left\{\begin{array}{crll}f((\frac{-v}{u})^{+})+O(|x|^{\epsilon}),&x\mapsto0^{+}\\
 f((\frac{-v}{u})^{-})+O(|x|^{\epsilon}),&x\mapsto0^{-}.\end{array}\right.$$
 We can easily verify that $f(\frac{x-v}{u})$ is normalized and that
 $$V_{\rb}\left(f(\frac{x-v}{u})e^{(\frac{1}{2}+b)|x|}\right)<\infty.$$
 Moreover, one obtains
\begin{eqnarray}
\sum_{\rho}&e^{v(\rho-\frac{1}{2})}&\int_{-\infty}^{+\infty}f(\frac{x}{u})e^{x(\rho-\frac{1}{2})}dx\nonumber\\
&=&-\sum_{n=2}^{+\infty}\frac{\Lambda_{F}(n)}{\sqrt{n}}f\left(\frac{\log n-v}{u}\right)-\sum_{n=2}^{+\infty}\frac{\overline{\Lambda_{F}(n)}}{\sqrt{n}}f\left(\frac{-\log n-v}{u}\right)\\
&&+\ m_{F}\left(\int_{-\infty}^{+\infty}f\left(\frac{x-v}{u}\right)e^{\frac{x}{2}}dx+\int_{-\infty}^{+\infty}f\left(\frac{x-v}{u}\right)e^{-\frac{x}{2}}dx\right)\nonumber\\
&&+\ 2f\left(\frac{-v}{u}\right)\log Q+\ \sum_{j=1}^{r}\lambda_{j}\left(\frac{\Gamma'}{\Gamma}(\frac{\lambda_{j}}{2}+\mu_{j})+\frac{\Gamma'}{\Gamma}(\frac{\lambda_{j}}{2}+\overline{\mu_{j}})\right)f\left(\frac{-v}{u}\right)\nonumber\\
&&-\ \sum_{j=1}^{r}\lambda_{j}\int_{0}^{+\infty}\left( f\left(\frac{-\lambda_{j}x-v}{u}\right)-f\left((\frac{-v}{u})^{-}\right)\right)\frac{e^{-(\frac{\lambda_{j}}{2}+\mu_{j})x}}{1-e^{-x}}dx\nonumber\\
&&-\ \sum_{j=1}^{r}\lambda_{j}\int_{0}^{+\infty}\left( f\left(\frac{\lambda_{j}x-v}{u}\right)-f\left((\frac{-v}{u})^{+}\right)\right)\frac{e^{-(\frac{\lambda_{j}}{2}+\overline{\mu_{j}})x}}{1-e^{-x}}dx.\nonumber
 \end{eqnarray}
Formula (3) is valid for  $v=0$ and all $0<u<1$ under conditions $a),\ b)$ and $c)$. Furthermore it is valid for all $0<u<1$ and $v\in{\rb}$ under conditions $a),\ b)$ and $c')$ given by: there exist $D>0$ and $\epsilon>0$ such that for all $a\in{\rb}$
$$\left\{\begin{array}{crll}|f(a+x)-f(a^{+})|\leq D|x|^{\epsilon},&x>0\\ |f(a+x)-f(a^{-})|\leq D|x|^{\epsilon'},&x<0.\end{array}\right. $$
The second result of this paper   gives an asymptotic to the left-hand side of (3) for fixed $v$ and $u\mapsto0^{+}$. It is stated in the following Theorem.

\begin{theorem}\label{thm:2} Let $f$ be some complex-valued function on $\rb$ satisfying the above conditions a), b) and c). Then, when $u\mapsto0^{+}$, one has
$$
\sum_{\rho}e^{v(\rho-\frac{1}{2})}\int_{-\infty}^{+\infty}f(\frac{x}{u})e^{x(\rho-\frac{1}{2})}dx=\left\{\begin{array}{crll}O(u)&\hbox{if}&v\neq\pm\log m,\ m=1,2,..\\-\frac{\Lambda_{F}(m)}{\sqrt{m}}f(0)+O(u)&\hbox{if}&v=\log m,\ m=2,3,..\\
-\frac{\overline{\Lambda_{F}(m)}}{\sqrt{m}}f(0)+O(u)&\hbox{if}&v=-\log m,\ m=2,3,..\end{array}\right.
$$
When $v=0$, one gets
\begin{eqnarray}
\sum_{\rho}\int_{-\infty}^{+\infty}f\left(\frac{x}{u}\right)e^{x(\rho-\frac{1}{2})}dx&=&-2f(0)\log Q \nonumber\\
&&\ +f(0^{-})\sum_{j=1}^{r}\lambda_{j}\left(\log\left[\left(\frac{\lambda_{j}}{2}+\mu_{j}\right)u\right]-\frac{\Gamma'}{\Gamma}
\left(\frac{\lambda_{j}}{2}+\mu_{j}\right)\right)\nonumber\\
&&+f(0^{+})\sum_{j=1}^{r}\lambda_{j}\left((\log\left[\left(\frac{\lambda_{j}}{2}+\overline{\mu_{j}}\right)u\right]-\frac{\Gamma'}{\Gamma}\left(\frac{\lambda_{j}}{2}+\overline{\mu_{j}}\right)\right)\nonumber\\
&&+ f(0)\sum_{j=1}^{r}\lambda_{j}\left(\frac{\Gamma'}{\Gamma}\left(\frac{\lambda_{j}}{2}+\mu_{j}\right)-\frac{\Gamma'}{\Gamma}\left(\frac{\lambda_{j}}{2}+\overline{\mu_{j}}\right)\right)\nonumber\\
&&-\sum_{j=1}^{r}\lambda_{j}\int_{0}^{+\infty}\left(f(-\lambda_{j}x)-f(0^{-})e^{-\left(\frac{\lambda_{j}}{2}+\mu_{j}\right)x}\right)\frac{dx}{x}\nonumber\\
&&-\sum_{j=1}^{r}\lambda_{j}\int_{0}^{+\infty}\left(f(\lambda_{j}x)-f(0^{+})e^{-\left(\frac{\lambda_{j}}{2}+\overline{\mu_{j}}\right)x}\right)\frac{dx}{x}+O(u),\nonumber
\end{eqnarray}
where $\Lambda_{F}(m)=b(n)\log n$ is the generalized  von Mangoldt function.
\end{theorem}
\noindent Moreover, if $f$ is continuous at 0, the last formula can be written as follows
\begin{eqnarray}
\sum_{\rho}\int_{-\infty}^{+\infty}&f\left(\frac{x}{u}\right)&e^{x(\rho-\frac{1}{2})}dx\nonumber\\&=&f(0)\left[-2\log Q +\sum_{j=1}^{r}\lambda_{j}\left(\log\left|\left(\frac{\lambda_{j}}{2}+\mu_{j}\right)u^{2}\right|\right)\right]\nonumber\\
&&-\sum_{j=1}^{r}2\lambda_{j}\int_{0}^{+\infty}\left(\frac{f(\lambda_{j}x)+f(-\lambda_{j}x)}{2}-f(0)e^{-\left(\frac{\lambda_{j}}{2}+\Re(\mu_{j})\right)x}
\cos\left(\Im(\mu_{j})x\right)\right)\frac{dx}{x}+O(u).\nonumber
\end{eqnarray}

\noindent Theorem \ref{thm:2} says that the quantity $$\sum_{\rho}e^{v(\rho-\frac{1}{2})}\int_{-\infty}^{+\infty}f(\frac{x}{u})e^{x(\rho-\frac{1}{2})}dx$$ behaves quite differently whether $v$
is in $\{0\}\cup \{± \log p^{l}; p:\hbox{prime},\  l:\hbox{positive integer}\}$ or not.

Now, we  give an expression for the error term in  Theorem \ref{thm:2} under the Riemann hypothesis.
\begin{theorem}\label{thm:3}
Let f be some complex-valued function on $\rb$ satisfying the conditions :\\
a) $f$ is normalized.\\
 b) There exists a constant $b>0$ such that
 $$V_{\rb}\left(f(x)e^{(\frac{1}{2}+b)|x|}\right)<\infty,$$
where $V_{\rb}(.)$ means the total variation on $\rb$.\\
 c) There exists a constant $\epsilon>0$ such that
 $$f(x)=\left\{\begin{array}{crll}f(x^{+})+O(|x|^{\epsilon}),&x\mapsto0^{+},\\
 f(x^{-})+O(|x|^{\epsilon}),&x\mapsto0^{-}.\end{array}\right.$$
 d) For the Fourier transform $\widehat{f}$ of f defined by
 $$\widehat{f}(t)=\int_{-\infty}^{+\infty}f(x)e^{-itx}dx,$$
 it holds that
 $$\int_{-\infty}^{+\infty}\left|\widehat{f}(t)\right|\left|\log|t|\right|dt<\infty.$$
Let $u$ be such that $0<u<1$. Then the error term in theorem \ref{thm:2}, for the case $v=0$, can be expressed in the form
$$-u\int_{0}^{+\infty}\left(N_{F}(T)-\frac{d_{F}}{2\pi}T\log\left(\frac{T}{2\pi }\right)-\frac{1}{2\pi}T\left(\log q_{F}-d_{F}\right)\right)\frac{d}{dT}\left(\widehat{f}(-u T)+\widehat{f}(u T)\right)dT,$$
where $N_{F}(T)$ is the number of non-trivial zeros $\rho = 1/2+ i\gamma$ with $ 0\leq t\leq T$ counted with multiplicities.
\end{theorem}
We finish this section by giving the following Lemma.
\begin{lemma}\label{lem:1}
Let f be some complex-valued function on $\rb$ satisfying conditions a), b), c) and d) of Theorem \ref{thm:3}.
 Then
 $$\frac{1}{\pi\lambda}\int_{0}^{+\infty}\widehat{f}\left(-\frac{t}{\lambda}\right)\log tdt=-\int_{0}^{+\infty}\left(\frac{f(\lambda x)+f(-\lambda x)}{2}-f(0)e^{-x}\right)\frac{dx}{x}.$$
\end{lemma}
\begin{proof}
We evaluate the following double integral
$$I=\frac{1}{\pi}\int_{0}^{+\infty}\left(\int_{0}^{+\infty}f(\lambda y)\frac{x}{x^{2}+y^{2}}dy-\frac{\pi}{2}f(\lambda x)\right)\frac{dx}{x}$$
by two ways. The first way is to use the equation
$$\frac{2}{\pi}\int_{0}^{+\infty}\frac{x}{x^{2}+y^{2}}dy=1.$$
From this and after some calculations, it follow that $ I=0$. The second way is to use the equation
$$\frac{1}{2}\int_{-\infty}^{+\infty}e^{-x|t|-ity}dt=\frac{x}{x^{2}+y^{2}},\ \ \ x>0.$$
From this it follow that
$$I=-\frac{1}{\pi\lambda}\int_{0}^{+\infty}\widehat{f}\left(-\frac{t}{\lambda}\right)\log tdt-\int_{0}^{+\infty}\left(\frac{f(\lambda x)+f(-\lambda x)}{2}-f(0)e^{-x}\right)\frac{dx}{x}.$$
 \end{proof}
\section{Lemmas and proof of the Theorem \ref{thm:1}}\label{s:2}
First, we give an explicit formula which connects certain sums involving $\rho$ with sums involving prime numbers as in Proposition \ref{prop:1}.
      \begin{lemma}\label{lem:2} Let $u>0,\ v\in{\rb}$. We have
 \begin{eqnarray}
 \sum_{\rho}e^{u\rho^{2}-v\rho}&=&m_{F}e^{u-v}+\frac{2\log Q}{\sqrt{4\pi u}}e^{-\frac{v^{2}}{4u}}-1\nonumber\\
 &&-\ \frac{1}{\sqrt{4\pi u}}\sum_{n=2}^{+\infty}\Lambda_{F}(n)e^{-\frac{(v+\log n)^{2}}{4u}}-\frac{1}{\sqrt{4\pi u}}\sum_{n=2}^{+\infty}\frac{\overline{\Lambda_{F}(n)}}{n}e^{-\frac{(v-2u-\log n)^{2}}{4u}}\nonumber\\
 &&+\ \frac{e^{\frac{u}{4}-\frac{v}{2}}}{\pi}\sum_{j=1}^{r}\lambda_{j}\int_{-\infty}^{+\infty}\log\left|(\frac{\lambda_{j}}{2}+\mu_{j})+i\lambda_{j}t\right|e^{-ut^{2}+it(u-v)}dt\nonumber\\
 &&-\ \sum_{j=1}^{r}\lambda_{j}I(\lambda_{j},\mu_{j}),
 \end{eqnarray}
 where
 \begin{eqnarray}
 I(\lambda,\mu)
 &=&\int_{0}^{+\infty}\left(\frac{1}{e^{x}-1}-\frac{1}{x}+1\right)e^{-\mu x}\frac{e^{-(v+\lambda x)^{2}/4u}}{\sqrt{4\pi u}}dx\nonumber\\
&&+\int_{0}^{+\infty}\left(\frac{1}{e^{x}-1}-\frac{1}{x}+1\right)e^{-(\lambda+\overline{\mu} x)}\frac{e^{-(v-\lambda x)^{2}/4u}}{\sqrt{4\pi u}}dx\nonumber
\end{eqnarray}
and $\rho$ runs over all non-trivial zeros  of $F(s)$ counted with multiplicity.
 \end{lemma}
\begin{proof} Since (see. \cite[page. 83]{8})
$$\int_{0}^{+\infty}\frac{1}{\sqrt{4\pi u}}e^{-(v+\log x)^{2}/4u}x^{s}\frac{dx}{x}= e^{us^{2}-vs},$$
then, by the Mellin inversion formula, we get
$$\frac{1}{\sqrt{4\pi u}}e^{-(v+\log x)^{2}/4u}=\frac{1}{2i\pi}\int_{\Re(s)=1+\delta}e^{us^{2}-vs}x^{-s}ds.$$
We have
\begin{equation}
\frac{1}{2i\pi}\int_{1+\delta-i\infty}^{1+\delta+i\infty}\frac{F'}{F}(s)e^{us^{2}-vs}ds=-\sum_{n=2}^{+\infty}\Lambda_{F}(n)\frac{1}{\sqrt{4\pi u}}e^{-(v+\log n)^{2}/4u}.
\end{equation}
Therefore
\begin{eqnarray}
\frac{1}{2i\pi}\int_{1+\delta-i\infty}^{1+\delta+i\infty}\frac{F'}{F}(s)e^{us^{2}-vs}ds&=&\frac{1}{2i\pi}\int_{-\delta-i\infty}^{-\delta+i\infty}\frac{F'}{F}(s)e^{us^{2}-vs}ds\nonumber\\
&&\ +\ \sum_{\rho}e^{u\rho^{2}-v\rho}-m_{F}e^{u-v}.
\end{eqnarray}
To prove (6), let $T>1$ be a  real number not imaginary part of any zero of $F$ and  consider the rectangle ${\mathcal L}$  which connects the points $ 1+\delta -iT,\ 1+\delta+iT,\ -\delta+iT$ and $-\delta-iT$ . By the residue theorem we obtain
$$
\int_{{\mathcal L}}\frac{F'}{F}(s)e^{us^{2}-vs}ds=\sum_{\rho;\ |Im(\rho)|\leq T}e^{u\rho^{2}-v\rho}-m_{F}e^{u-v}.
$$
 Then
\begin{eqnarray}
\frac{1}{2i\pi}\int_{1+\delta-iT}^{1+\delta+iT}\frac{F'}{F}(s)e^{us^{2}-vs}ds&=&
\frac{1}{2i\pi}\left\{\int_{-\delta-iT}^{-\delta+iT}+\int_{-\delta+iT}^{1+\delta+iT}+\int_{1+\delta-iT}^{-\delta-iT}\right\}
\frac{F'}{F}(s)e^{us^{2}-vs}ds\nonumber\\
&&+\ \sum_{\rho;\ |Im(\rho)|\leq T}e^{u\rho^{2}-v\rho}-m_{F}e^{u-v}.\nonumber
\end{eqnarray}
For $T\rightarrow+\infty$, we have $\frac{F'}{F}(\sigma+iT)\ll_{F} \log^{2}T$
uniformly for $\sigma\in{[-1,2]}$  \cite[Equation (3.3)]{12}.
The integrals along horizontal lines tend to 0 as $T\rightarrow+\infty$, indeed,
\begin{eqnarray}
\left|\frac{1}{2i\pi}\int_{-\delta+iT}^{1+\delta+iT}\frac{F'}{F}(s)e^{us^{2}-vs}ds\right|
&\ll_{F}& (\log T)^{2}\int_{-\delta}^{1+\delta}e^{u(\sigma^{2}-T^{2})-v\sigma}d\sigma\nonumber\\
&\ll_{F}& e^{-uT^{2}}(\log T)^{2}\int_{-\delta}^{1+\delta}e^{u\sigma^{2}-v\sigma}d\sigma\nonumber\\
&\ll_{F,\delta}& e^{-uT^{2}}(\log T)^{2}.\nonumber
\end{eqnarray}
 The integral over the lower horizontal boundary is treated analogously. Furthermore, the integrals along vertical lines converge to the integrals appearing in formula (6). Since the series over $\rho$ have exponential decay in $T$ as $T\rightarrow+\infty$, the series is convergent and
 $$\lim_{T\mapsto +\infty}\sum_{\rho;\ |Im(\rho)|\leq T}e^{u\rho^{2}-v\rho}=\sum_{\rho}e^{u\rho^{2}-v\rho}.$$
By the functional equation of $F(s)$, the first term of the  right-hand side of (6) is
\begin{eqnarray}
&&\frac{1}{2i\pi}\int_{-\delta-i\infty}^{-\delta+i\infty}\left[-\frac{\overline{F}'}{\overline{F}}(1-s)-2\log Q\right]e^{us^{2}-vs}ds\nonumber\\&-&\sum_{j=1}^{r}\lambda_{j}\frac{1}{2i\pi}\int_{-\delta-i\infty}^{-\delta+i\infty}\left(\frac{\Gamma'}{\Gamma}(\lambda_{j}s+\mu_{j})+\frac{\Gamma'}{\Gamma}(\lambda_{j}(1-s)+\overline{\mu}_{j})\right)e^{us^{2}-vs}ds\nonumber\\
&=&\frac{1}{\sqrt{4\pi u}}\sum_{n=2}^{+\infty}\frac{\overline{\Lambda_{F}(n)}}{n}\frac{1}{\sqrt{4\pi u}}e^{-(v-2u-\log n)^{2}/4u}-\frac{2\log Q}{\sqrt{4\pi u}}e^{-v^{2}/4u}+1\nonumber\\
&&-\sum_{j=1}^{r}\lambda_{j}\frac{1}{2\pi}\int_{-\infty}^{+\infty}\left(\frac{\Gamma'}{\Gamma}(\frac{\lambda_{j}}{2}+\mu_{j}+i\lambda_{j}t)+\frac{\Gamma'}{\Gamma}(\frac{\lambda_{j}}{2}+\overline{\mu}_{j}-i\lambda_{j}t)\right)e^{u(\frac{1}{2}+it)^{2}-v(\frac{1}{2}+it)}dt.\nonumber
\end{eqnarray}
Hence and from (5) we have
\begin{eqnarray}
 \sum_{\rho}&e^{u\rho^{2}-v\rho}&\nonumber\\
&=&m_{F}e^{u-v}+\frac{2\log Q}{\sqrt{4\pi u}}e^{-\frac{v^{2}}{4u}}-1\nonumber\\
 &&- \frac{1}{\sqrt{4\pi u}}\sum_{n=2}^{+\infty}\Lambda_{F}(n)e^{-\frac{(v+\log n)^{2}}{4u}}-\frac{1}{\sqrt{4\pi u}}\sum_{n=2}^{+\infty}\frac{\overline{\Lambda_{F}(n)}}{n}e^{-\frac{(v-2u-\log n)^{2}}{4u}}\nonumber\\
&&+ \sum_{j=1}^{r}\lambda_{j}\frac{e^{\frac{u}{4}-\frac{v}{2}}}{2\pi}\int_{-\infty}^{+\infty}\left(\frac{\Gamma'}{\Gamma}(\frac{\lambda_{j}}{2}
+\mu_{j}+i\lambda_{j}t)+\frac{\Gamma'}{\Gamma}(\frac{\lambda_{j}}{2}+\overline{\mu}_{j}-i\lambda_{j}t)\right)e^{-ut^{2}+it(u-v)}dt.\nonumber\\
\end{eqnarray}
Denote by $L$ the expression
$$\sum_{j=1}^{r}\lambda_{j}\frac{e^{\frac{u}{4}-\frac{v}{2}}}{2\pi}\int_{-\infty}^{+\infty}\left(\frac{\Gamma'}{\Gamma}(\frac{\lambda_{j}}{2}+\mu_{j}+i\lambda_{j}t)+\frac{\Gamma'}{\Gamma}(\frac{\lambda_{j}}{2}+\overline{\mu}_{j}-i\lambda_{j}t)\right)e^{-ut^{2}+it(u-v)}dt. $$
Using the formula
$$\frac{\Gamma'}{\Gamma}(z)=\log z-\int_{0}^{+\infty}\left(\frac{1}{e^{x}-1}-\frac{1}{x}+1\right)e^{-zx}dx,\ \ Re(z)>0,$$
we obtain
\begin{eqnarray}
L&=&\frac{e^{\frac{u}{4}-\frac{v}{2}}}{2\pi}\sum_{j=1}^{r}\lambda_{j}\int_{-\infty}^{+\infty}\left(\frac{\Gamma'}{\Gamma}(\frac{\lambda_{j}}{2}+\mu_{j}+i\lambda_{j}t)+\frac{\Gamma'}{\Gamma}(\frac{\lambda_{j}}{2}+\overline{\mu}_{j}-i\lambda_{j}t)\right)e^{-ut^{2}+it(u-v)}dt\nonumber\\
&=&\frac{e^{\frac{u}{4}-\frac{v}{2}}}{\pi}\sum_{j=1}^{r}\lambda_{j}\int_{-\infty}^{+\infty}\log\left|(\frac{\lambda_{j}}{2}+\mu_{j})+i\lambda_{j}t\right|e^{-ut^{2}+it(u-v)}dt\nonumber\\
&&-\frac{e^{\frac{u}{4}-\frac{v}{2}}}{2\pi}\sum_{j=1}^{r}\lambda_{j}\int_{0}^{+\infty}\left(\frac{1}{e^{x}-1}-\frac{1}{x}+1\right)\left[e^{-(\frac{\lambda_{j}}{2}+\mu_{j})x}\int_{-\infty}^{+\infty}e^{-ut^{2}+it(u-v-\lambda_{j}x)}dt\right]dx\nonumber\\
&&-\frac{e^{\frac{u}{4}-\frac{v}{2}}}{2\pi}\sum_{j=1}^{r}\lambda_{j}\int_{0}^{+\infty}\left(\frac{1}{e^{x}-1}-\frac{1}{x}+1\right)\left[e^{-(\frac{\lambda_{j}}{2}+\overline{\mu}_{j})x}\int_{-\infty}^{+\infty}e^{-ut^{2}+it(u-v+\lambda_{j}x)}dt\right]dx.\nonumber
 \end{eqnarray}
Note that
$$\int_{-\infty}^{+\infty}e^{-ut^{2}+it(u-v-\lambda_{j}x)}dt=\sqrt{\frac{\pi}{u}}\quad e^{-(u-v-\lambda_{j}x)^{2}/4u}=\sqrt{\frac{\pi}{u}}\quad e^{-\frac{u}{4}+\frac{v}{2}+\frac{\lambda_{j}}{2}x}\ e^{-(v+\lambda_{j}x)^{2}/4u}$$
and
$$\int_{-\infty}^{+\infty}e^{-ut^{2}+it(u-v+\lambda_{j}x)}dt=\sqrt{\frac{\pi}{u}}\quad e^{-(u-v+\lambda_{j}x)^{2}/4u}=\sqrt{\frac{\pi}{u}}\quad e^{-\frac{u}{4}+\frac{v}{2}-\frac{\lambda_{j}}{2}x}\ e^{-(v-\lambda_{j}x)^{2}/4u}.$$
Finally, by considering  (7) and the above equations, Lemma \ref{lem:2} is proved.
\end{proof}

\noindent From the next Lemma given below,  we will deduce  Theorem \ref{thm:1}.
\begin{lemma}\label{lem:3} For $0<u<1$, we have\\

$\displaystyle{\int_{-\infty}^{+\infty}\log\left|(\frac{\lambda}{2}+\mu)+i\lambda t\right|e^{-ut^{2}+it(u-v)}dt}$\\

$\displaystyle{=\left\{\begin{array}{crll}\hbox{O}\left(\frac{1}{|u-v|^{2}}\right)&\hbox{if}&v\neq0\ \hbox{and}\  v\neq u,\\ \sqrt{\frac{\pi}{4u}}\left(\log \frac{1}{u}-\gamma_{0}\right)+\sqrt{\frac{\pi}{u}}\log(\frac{\lambda}{2})+\hbox{O}(1)&\hbox{if}&u=v\ \hbox{or}\ v=0,\end{array}\right.}$\\

\noindent where $\hbox{O}(1)$ is a constant depending on $\mu$.
\end{lemma}
\begin{proof}.\  Without loss of generality, we assume that $\mu$ is a real number.\\
$\bullet$ In the case  $u\neq v$ and $v\neq0$, we have
\begin{eqnarray}
\int_{-\infty}^{+\infty}&&\log\left|(\frac{\lambda}{2}+\mu)+i\lambda t\right|\ e^{-ut^{2}+it(u-v)}dt\nonumber\\
&=& \int_{0}^{+\infty}\log\left((\frac{\lambda}{2}+\mu)^{2}+\frac{\lambda^{2}\ t^{2}}{u}\right)e^{-t^{2}}\cos\left(\frac{t(u-v)}{\sqrt{u}}\right)\frac{dt}{\sqrt{u}}\nonumber\\
&=&\ -\frac{1}{u-v}\int_{0}^{+\infty}\frac{d}{dt}\left[\log\left((\frac{\lambda}{2}+\mu)^{2}+\frac{\lambda^{2}\ t^{2}}{u}\right)e^{-t^{2}}\right]\sin\left(\frac{t(u-v)}{\sqrt{u}}\right)dt\nonumber\\
&=&\ -\frac{\sqrt{u}}{(u-v)^{2}}\int_{0}^{+\infty}\frac{d^{2}}{dt^{2}}\left[\log\left((\frac{\lambda}{2}+\mu)^{2}+\frac{\lambda^{2}\ t^{2}}{u}\right)e^{-t^{2}}\right]\cos\left(\frac{t(u-v)}{\sqrt{u}}\right)dt\nonumber\\
&=&\ -\frac{\lambda\sqrt{u}}{(u-v)^{2}}\left\{M_{1}+M_{2}+M_{3}\right\},
\end{eqnarray}
where
\begin{eqnarray}M_{1}&=&\int_{0}^{+\infty}\frac{2\lambda\left(-\lambda^{2}\ t^{2}+ u(\frac{\lambda}{2}+\mu)^{2}\right)}{\left(\lambda^{2}\ t^{2}+u(\frac{\lambda}{2}+\mu)^{2}\right)^{2}}\ e^{-t^{2}}\cos\left(\frac{t(u-v)}{\sqrt{u}}\right)dt,\nonumber\\
M_{2}&=&-\int_{0}^{+\infty}\frac{8\lambda\ t^{2}}{\lambda^{2}\ t^{2}+u(\frac{\lambda}{2}+\mu)^{2}}\ e^{-t^{2}}\cos\left(\frac{t(u-v)}{\sqrt{u}}\right)dt\nonumber\\
\hbox{and}&&\nonumber\\
M_{3}&=&\int_{0}^{+\infty}\log\left((\frac{\lambda}{2}+\mu)^{2}+\frac{\lambda^{2}\ t^{2}}{u}\right)e^{-t^{2}}(4t^{2}-2)\cos\left(\frac{t(u-v)}{\sqrt{u}}\right)dt.\nonumber
\end{eqnarray}
We have
\begin{eqnarray}
|M_{1}|&\leq&2\int_{0}^{+\infty}\frac{e^{-t^{2}}}{\lambda^{2}\ t^{2}+u(\frac{\lambda}{2}+\mu)^{2}}dt
=\frac{2}{\lambda^{2}}\int_{0}^{+\infty}\frac{e^{-t^{2}}}{ t^{2}+u(\frac{1}{2}+\frac{\mu}{\lambda})^{2}}dt\nonumber\\
&\leq&\frac{2}{\lambda^{2}}\ \frac{1}{u(\frac{1}{2}+\frac{\mu}{\lambda})^{2}}\int_{0}^{\sqrt{u}}dt+2\int_{\sqrt{u}}^{+\infty}\frac{dt}{t^{2}}\nonumber\\
&\ll&\frac{1}{\sqrt{u}},
\end{eqnarray}
\begin{equation}
|M_{2}|\leq8\int_{0}^{+\infty}e^{-t^{2}}dt=4\sqrt{\pi}
\end{equation}
and
\begin{eqnarray}
|M_{3}|&\leq&\int_{0}^{+\infty}\left|\log\left((\frac{\lambda}{2}+\mu)^{2}+\frac{\lambda^{2}\ t^{2}}{u}\right)\right|e^{-t^{2}}(4t^{2}+2)dt\nonumber\\
&\leq&\log\frac{1}{u}\int_{0}^{+\infty}e^{-t^{2}}(4t^{2}+2)dt\nonumber\\
&&+ \int_{0}^{+\infty}\left|\log\left(u(\frac{\lambda}{2}+\mu)^{2}+\lambda^{2}\ t^{2}\right)\right|e^{-t^{2}}(4t^{2}+2)dt\nonumber\\
&\ll&\log\frac{1}{u}+\int_{1/\lambda_{j}}^{+\infty}\log t\ e^{-t^{2}}(4t^{2}+2)dt\nonumber\\
&\ll&\log\frac{1}{u}.
\end{eqnarray}
From (8), (9), (10) and (11), we obtain
\begin{equation}
\int_{-\infty}^{+\infty}\log\left|(\frac{\lambda}{2}+\mu)+i\lambda t\right|e^{-ut^{2}+it(u-v)}dt=O\left(\frac{1}{(u-v)^{2}}\right).
\end{equation}
$\bullet$ In the case $u=v$, we have
\begin{eqnarray}
\int_{-\infty}^{+\infty}&&\log\left|(\frac{\lambda}{2}+\mu)+i\lambda t\right|\ e^{-ut^{2}}dt\nonumber\\&=&\int_{0}^{+\infty}\log\left((\frac{\lambda}{2}+\mu)^{2}+\frac{\lambda^{2}\ t^{2}}{u}\right)e^{-t^{2}}\frac{dt}{\sqrt{u}}\nonumber\\
&=&\frac{\log\frac{1}{u}}{\sqrt{u}}\int_{0}^{+\infty}e^{-t^{2}}dt+\frac{1}{\sqrt{u}}\int_{0}^{+\infty}\log
\left(u(\frac{\lambda}{2}+\mu)^{2}+\lambda^{2}\ t^{2}\right)e^{-t^{2}}dt\nonumber\\
&=&\frac{\sqrt{\pi}}{2}\frac{\log\frac{1}{u}}{\sqrt{u}}+\frac{1}{\sqrt{u}}\left\{N_{1}+N_{2}\right\},
\end{eqnarray}
 with $$N_{1}=\int_{0}^{+\infty}\log(\lambda^{2}\ t^{2})e^{-t^{2}}dt$$
 and
 $$
N_{2}=\int_{0}^{+\infty}\log\left(1+\frac{u(\frac{1}{2}+\frac{\mu}{\lambda})^{2}}{t^{2}}\right)e^{-t^{2}}dt.$$
We have
\begin{eqnarray}
|N_{1}|&=&\int_{0}^{+\infty}\log(\lambda^{2})\ e^{-t^{2}}dt+\int_{0}^{+\infty}\log( t^{2})\ e^{-t^{2}}dt\nonumber\\
&=&2\log\lambda\int_{0}^{+\infty}e^{-t^{2}}dt+2\int_{0}^{+\infty}\log t\ e^{-t}\frac{dt}{2\sqrt{t}}\nonumber\\
&=&\sqrt{\pi}\log \lambda+\frac{1}{2}\Gamma'(\frac{1}{2})\nonumber\\
&=&\sqrt{\pi}\log\left(\frac{\lambda}{2}\right)-\frac{\sqrt{\pi}}{2}\gamma_{0}.
\end{eqnarray}
 Furthermore,
\begin{eqnarray}
|N_{2}|&=&\int_{0}^{\sqrt{u}}\log\left(1+\frac{u(\frac{1}{2}+\frac{\mu}{\lambda})^{2}}{t^{2}}\right)e^{-t^{2}}dt+\int_{\sqrt{u}}^{+\infty}\log\left(1+\frac{u(\frac{1}{2}+\frac{\mu}{\lambda})^{2}}{t^{2}}\right)dt\nonumber\\
&\leq&\int_{0}^{\sqrt{u}}\log\left(1+\frac{u(\frac{1}{2}+\frac{\mu}{\lambda})^{2}}{t^{2}}\right)e^{-t^{2}}dt+u(\frac{1}{2}+\frac{\mu}{\lambda})^{2}\int_{\sqrt{u}}^{+\infty}\frac{e^{-t^{2}}}{t^{2}}dt\nonumber\\
&&=\ \sqrt{u}\ \log\left(1+(\frac{1}{2}+\frac{\mu}{\lambda})^{2}\right)+\int_{0}^{\sqrt{u}}\frac{2u(\frac{1}{2}+\frac{\mu}{\lambda})^{2}}{t^{2}+u(\frac{1}{2}+\frac{\mu}{\lambda})^{2}}dt\nonumber\\
&&+\ u(\frac{1}{2}+\frac{\mu}{\lambda})^{2}\int_{\sqrt{u}}^{+\infty}\frac{e^{-t^{2}}}{t^{2}}dt\nonumber\\
&&=\ \sqrt{u}\ \log\left(1+(\frac{1}{2}+\frac{\mu}{\lambda})^{2}\right)+2\sqrt{u}\ (\frac{1}{2}+\frac{\mu}{\lambda})^{2}+u(\frac{1}{2}+\frac{\mu}{\lambda})^{2}\int_{\sqrt{u}}^{+\infty}\frac{e^{-t^{2}}}{t^{2}}dt\nonumber\\
&&\ll\ \sqrt{u}.
\end{eqnarray}
Using equations (13), (14) and (15), we get\begin{equation}\int_{-\infty}^{+\infty}\log\left|(\frac{\lambda}{2}+\mu)+i\lambda t\right|e^{-ut^{2}}dt
=\sqrt{\frac{\pi}{4u}}\left(\log\frac{1}{u}-\gamma_{0}\right)+\ \sqrt{\frac{\pi}{u}}\ \log\left(\frac{\lambda}{2}\right)+O(1).
\end{equation}
$\bullet$ In the case $v=0$, we have
\begin{eqnarray}
\int_{-\infty}^{+\infty}&&\log\left|(\frac{\lambda}{2}+\mu)+i\lambda t\right|\ e^{-ut^{2}+itu}dt\nonumber\\&=&\int_{0}^{+\infty}\log\left((\frac{\lambda}{2}+\mu)^{2}+\frac{\lambda^{2}\ t^{2}}{u}\right)e^{-t^{2}}\cos(\sqrt{u}t)\frac{dt}{\sqrt{u}}\nonumber\\
&=&\frac{\log\frac{1}{u}}{\sqrt{u}}\int_{0}^{+\infty}e^{-t^{2}}\cos(\sqrt{u}t)dt+\frac{1}{\sqrt{u}}\int_{0}^{+\infty}\log\left(\lambda^{2}\ t^{2}\right)e^{-t^{2}}\cos(\sqrt{u}t)dt\nonumber\\
&&+\ \frac{1}{\sqrt{u}}\int_{0}^{+\infty}\log\left(1+\frac{u}{t^{2}}(\frac{1}{2}+\frac{\mu}{\lambda})\right)e^{-t^{2}}\cos(\sqrt{u}t)dt\nonumber\\
&=&\frac{\log\frac{1}{u}}{\sqrt{u}}\ P_{1}+\frac{1}{\sqrt{u}}\{P_{2}+P_{3}\}.
\end{eqnarray}
With the equality  $$\cos(\sqrt{u}t)=1+O(ut^{2}),$$
we obtain
\begin{equation}
P_{1}=\int_{0}^{+\infty}e^{-t^{2}}dt+O\left(\int_{0}^{+\infty}t^{2}e^{-t^{2}}dt\right)=\frac{\sqrt{\pi}}{2}+O(u)
\end{equation}
and
 \begin{eqnarray}P_{2}&=&\int_{0}^{+\infty}\log\left(\lambda^{2}\ t^{2}\right)e^{-t^{2}}dt+O\left(u\int_{0}^{+\infty}\log\left(\lambda^{2}\ t^{2}\right)t^{2}e^{-t^{2}}dt\right)\nonumber\\
 &=&\sqrt{\pi}\log\left(\frac{\lambda}{2}\right)-\frac{\sqrt{\pi}}{2}\gamma_{0}+O_{F}(u).
 \end{eqnarray}
  Furthermore, we have
 \begin{equation}
 |P_{3}|\leq N_{2}=O(\sqrt{u}).
 \end{equation}
 Hence, using  (15),\ (16),\  (17) and (18), we obtain
  \begin{equation}
\int_{-\infty}^{+\infty}\log\left|(\frac{\lambda}{2}+\mu)+i\lambda t\right|e^{-ut^{2}+itu}dt
=\sqrt{\frac{\pi}{4u}}\left(\log\frac{1}{u}-\gamma_{0}\right)+ \sqrt{\frac{\pi}{u}}\ \log\left(\frac{\lambda}{2}\right)+O(1),
 \end{equation}
and Lemma \ref{lem:3} follows.
\end{proof}
\begin{proof}\  {\bf (of  Theorem \ref{thm:1})}.
First, it is easy to verify that $$0\leq
\left(\frac{1}{e^{x}-1}-\frac{1}{x}+1\right)\leq1,$$ hence
$$0\leq I(\lambda,\mu)\leq2\int_{0}^{+\infty}\frac{e^{-(v+\lambda x)^{2}/4u}}{\sqrt{4\pi u}}dx\ll
1.$$
Second, we consider the asymptotic behavior of the quantity
\begin{equation}
\frac{1}{\sqrt{4\pi u}}\sum_{n=2}^{+\infty}\Lambda_{F}(n)e^{-\frac{(v+\log n)^{2}}{4u}}-\frac{1}{\sqrt{4\pi u}}\sum_{n=2}^{+\infty}\frac{\overline{\Lambda_{F}(n)}}{n}e^{-\frac{(v-2u-\log n)^{2}}{4u}},
\end{equation}
in Lemma \ref{lem:2}. The behavior depends on the choice of $v$.  For the case $v=0$ and $0<u<1$, (22) is of exponential decay as $u\rightarrow0^{+}$. For the case $v=-\log m,\ m\in{\nb\mid\{0,1\}}$, and $0<u<1$, we have
\begin{eqnarray}
&&\frac{1}{\sqrt{4\pi u}}\sum_{n=2}^{+\infty}\Lambda_{F}(n)e^{-\frac{(v+\log n)^{2}}{4u}}-\frac{1}{\sqrt{4\pi u}}\sum_{n=2}^{+\infty}\frac{\overline{\Lambda_{F}(n)}}{n}e^{-\frac{(v-2u-\log n)^{2}}{4u}}\nonumber\\&=&
-\frac{\Lambda_{F}(m)}{\sqrt{4\pi u}}+O\left(\frac{e^{-(\log 2)^{2}/8u}}{\sqrt{u}}\sum_{n\geq2}\frac{\Lambda_{F}(n)}{n}e^{-(\log n)^{2}/8}\right)\nonumber\\
&&+\ e^{-\frac{1}{8u}(-\log m+\log(m+1))}\left(\sum_{m\neq n=2}^{m^{2}}\Lambda_{F}(n)+\sum_{n>m^{2}}\Lambda_{F}(n)e^{-(\log n)^{2}/32}\right).\nonumber
\end{eqnarray}
Similarly  asymtotic behaviour for (22) can be obtained for other $v$. Combining this with Lemmas \ref{lem:2} and \ref{lem:3} and since $e^{\frac{u}{4}-\frac{v}{2}}=1$ for $u=v$ or $v=0$ as $u\rightarrow0^{+}$, we obtain
\begin{eqnarray}
\sum_{\rho}e^{u\rho^{2}-v\rho}&=&\frac{d_{F}}{\sqrt{16\pi u}}\left(\log\frac{1}{u}-\gamma_{0}\right)+\frac{1}{\sqrt{4\pi u}}\log\left(\frac{q_{F}}{(4\pi)^{d_{F}}}\right)+O_{F}(1), \ \ \hbox{if}\ \ u\rightarrow0^{+},\nonumber\\
&=&-\frac{d_{F}}{4\sqrt{\pi u}}\log u+O_{F}(1), \ \ \hbox{if}\ \ u\rightarrow0^{+},\nonumber
\end{eqnarray}
where $\gamma_{0}$ is the Euler constant. Hence Theorem \ref{thm:1} follows.\\
\end{proof}

 The above asymptotic formula  of $\sum_{\rho}e^{u\rho^{2}-v\rho}$ is another version of the asymptotic formula of $N_{F}(T)$, number of non-trivial zeros $\rho$ with $0<\gamma<T$. To see this, we consider the case $v=u$ in $(i)$ under the Generalized Riemann hypothesis, then the asymptotic formula in $(i)$ is
 $$ \sum_{\gamma}e^{-u(1/4+\gamma^{2})}=\frac{d_{F}}{\sqrt{16\pi u}}\left(\log\frac{1}{u}-\gamma_{0}\right)+\frac{1}{\sqrt{4\pi u}}\log\left(\frac{q_{F}}{(4\pi)^{d_{F}}}\right)+O_{F}(1)\ \hbox{if}\ \ u\rightarrow0^{+}.$$
 By integration by parts, it follows that
$$ -\int_{0}^{+\infty}N_{F}(T)d(e^{-uT^{2}})=\frac{d_{F}}{2\sqrt{16\pi u}}\left(\log\frac{1}{u}-\gamma_{0}\right)+\frac{1}{2\sqrt{4\pi u}}\log\left(\frac{q_{F}}{(4\pi)^{d_{F}}}\right)+O_{F}(1)\ \hbox{if}\ \ u\rightarrow0^{+}.$$
The asymptotic formula in $ii)$ may be regarded as a smooth version of (1) with the measure given by the Gaussian function.
\begin{section}{Proof of  Theorem \ref{thm:2}}
The proof  is an analogous of the argument used by Kamya in \cite[Theorem 1]{7}.\\
      \noindent{\bf Case: $0<u<1$ and $v\in{\rb^{*}}$}.\\
Let us denote the right-hand side of (3) by $T_{1}+T_{2}+m_{F}T_{3}+m_{F}T_{4}+T_{5}+T_{6}+T_{7}+T_{8}$.\\
We know that there exists a constant $c>0$ such that
\begin{equation}
|f(x)|\leq c\ e^{-(\frac{1}{2}+b)|x|}.
\end{equation}
This gives \begin{equation}T_{5},\ T_{6}=O(u), 0<u<1.
\end{equation}
From (23), we deduce that
\begin{eqnarray}
\left|\int_{-\infty}^{+\infty}f(\frac{x-v}{u})e^{\frac{x}{2}}dx\right|&\leq&c\int_{-\infty}^{+\infty}e^{-(\frac{1}{2}+b)|\frac{x-v}{u}|}e^{\frac{x}{2}}dx\nonumber\\
&=&c\ u\ e^{v/2}\int_{-\infty}^{+\infty}e^{-(\frac{1}{2}+b)|x|}e^{ux/2}dx\nonumber\\
&\leq&c\ u\ e^{v/2}\int_{-\infty}^{+\infty}e^{-b|x|}dx\nonumber\\
&=&O(u).\nonumber
\end{eqnarray}
Hence
\begin{equation}T_{3},\ T_{4}=O(u), 0<u<1.
\end{equation}
By using conditions $c')$, (23) and $0< x\leq |v|/2$, we can easily
prove that
\begin{equation}
\left|f(\frac{x-v}{u})-f((\frac{-v}{u})^{+})\right|\leq(2c D)^{1/2}\frac{e^{-(\frac{1}{2}+b)\frac{|v|}{4u}}}{u^{\epsilon/2}}x^{\epsilon/2}.
\end{equation}
Therefore
\begin{eqnarray}
|T_{7}|&\leq&\int_{0}^{|v|/2}\left( f(\frac{-\lambda_{j}x-v}{u})-f((\frac{-v}{u})^{+})\right)\frac{e^{-(\frac{\lambda}{2}+\mu)x}}{1-e^{-x}}dx\nonumber\\&&+\ \int_{|v|/2}^{+\infty}\left( f(\frac{-\lambda_{j}x-v}{u})-f((\frac{-v}{u})^{+})\right)\frac{e^{-(\frac{\lambda}{2}+\mu)x}}{1-e^{-x}}dx\nonumber\\
&\leq&(2c D)^{1/2}\frac{e^{-(\frac{1}{2}+b)\frac{|v|}{4u}}}{u^{\epsilon/2}}(\lambda_{j})^{\epsilon/2}\int_{0}^{|v|/2}x^{\epsilon/2}\frac{e^{-(\frac{\lambda}{2}+\mu)x}}{1-e^{-x}}dx\nonumber\\
&&+\ c\int_{|v|/2}^{+\infty}\left(e^{-(\frac{1}{2}+b)|x-v|/u}+e^{-(\frac{1}{2}+b)|v|/u}\right)\frac{e^{-(\frac{\lambda}{2}+\mu)x}}{1-e^{-x}}dx\nonumber\\
&\leq&\frac{c}{1-e^{|v|/2}}(\lambda_{j})^{\epsilon/2}\int_{|v|/2}^{+\infty}e^{-(\frac{\lambda}{2}+\mu)x}e^{-(\frac{1}{2}+b)|x-v|/u}dx+O(u).
\end{eqnarray}
For $0<u<1$, we have
\begin{eqnarray}
\int_{|v|/2}^{+\infty}e^{-(\frac{\lambda}{2}+\mu)x}e^{-(\frac{1}{2}+b)|x-v|/u}dx&\leq& e^{-v/2}\int_{-\infty}^{+\infty}e^{-(\frac{\lambda}{2}+\mu)x}e^{-(\frac{1}{2}+b)|x|/u}dx\nonumber\\
&=&u\ e^{-v/2}\int_{-\infty}^{+\infty}e^{-(\frac{\lambda}{2}+\mu)x}e^{-(\frac{1}{2}+b)|x|}dx\nonumber\\
&\leq&u\ e^{-v/2}\int_{-\infty}^{+\infty}e^{-b|x|}dx\nonumber\\
&=&O_{F}(u).\nonumber
\end{eqnarray}
Consequently, we obtain
\begin{equation}
T_{7}=O_{F}(u),\ 0<u<1.
\end{equation}
The same argument used above yields
\begin{equation}
T_{8}=O_{F}(u),\ 0<u<1.
\end{equation}
$\bullet$ If $v\neq\pm\log m,\ m=1,2,...$ and $0<u<1/2$,  by (23), we have
\begin{eqnarray}
&&|T_{1}+T_{2}|\leq c\sum_{n=2}^{+\infty}\frac{|\Lambda_{F}(n)|}{\sqrt{n}}e^{-(\frac{1}{2}+b)\frac{|\log n-v|}{u}}+c\sum_{n=2}^{+\infty}\frac{|\overline{\Lambda_{F}(n)}|}{\sqrt{n}}e^{-(\frac{1}{2}+b)\frac{|\log n+v|}{u}}\nonumber\\
&\leq&c\ e^{-(\frac{1}{2}+b)\delta/2u}\sum_{n=2}^{+\infty}\frac{|\Lambda_{F}(n)|}{\sqrt{n}}e^{-(\frac{1}{2}+b)\frac{|\log n-v|}{u}}+\ c\ e^{-(\frac{1}{2}+b)\delta/2u}\sum_{n=2}^{+\infty}\frac{|\overline{\Lambda_{F}(n)}|}{\sqrt{n}}e^{-(\frac{1}{2}+b)\frac{|\log n+v|}{u}}\nonumber\\
&\leq&
c\ e^{-(\frac{1}{2}+b)\delta/2u}\left\{\sum_{n=2}^{+\infty}\frac{|\Lambda_{F}(n)|}{\sqrt{n}}e^{-(\frac{1}{2}+b)|\log n-v|}+\sum_{n=2}^{+\infty}\frac{|\overline{\Lambda_{F}(n)}|}{\sqrt{n}}e^{-(\frac{1}{2}+b)|\log n+v|}\right\},\nonumber
\end{eqnarray}
where $\delta$ is the distance between the set $\{\pm\log m,\ m=1,2,...\}$ and $\{v\}$. We may assume that $v>0$ to estimate the right-hand side of the last expression. This gives
\begin{eqnarray}
|T_{1}+T_{2}|&\leq&c\ e^{-(\frac{1}{2}+b)\delta/2u}\left\{\sum_{2\leq n\leq e^{v}}\frac{|\Lambda_{F}(n)|}{\sqrt{n}}+e^{(\frac{1}{2}+b)v}\sum_{e^{v}<n}\frac{|\Lambda_{F}(n)|}{\sqrt{n}}e^{-(\frac{1}{2}+b)\log n}\right\}\nonumber\\
&&+\ c\ e^{-(\frac{1}{2}+b)\delta/2u}\sum_{n=2}^{+\infty}\frac{|\overline{\Lambda_{F}(n)}|}{\sqrt{n}}e^{-(\frac{1}{2}+b)\log n}\nonumber\\
&=&O\left(e^{-(\frac{1}{2}+b)\delta/2u}\left(1+\sum_{n=2}^{+\infty}\frac{\Lambda_{F}(n)}{n^{1+b}}\right)\right)\nonumber\\
&=&O(u).\nonumber
\end{eqnarray}
If $v=\log m,\ m=2,3,...$, we pick up the term $-\frac{\Lambda_{F}(m)}{\sqrt{m}}f(0)$ from $T_{1}+T_{2}$, the remainder sums are dominated by $O(u)$ with the same argument as above. While, for $v=-\log m,\ m=2,3,..$ we pick up again the term $-\frac{\overline{\Lambda_{F}(m)}}{\sqrt{m}}f(0)$ from $T_{1}+T_{2}$. Hence, for $u\mapsto0^{+}$,
\begin{equation}
T_{1}+T_{2}=\left\{\begin{array}{crll}O(u)&\hbox{if}&v\neq\pm\log m,\ m=1,2,..\\
-\frac{\Lambda_{F}(m)}{\sqrt{m}}f(0)+O(u)&\hbox{if}&v=\log m,\ m=2,3,..\\
-\frac{\overline{\Lambda_{F}(m)}}{\sqrt{m}}f(0)+O(u)&\hbox{if}&v=-\log m,\ m=2,3,..\end{array}\right.
\end{equation}
{\bf Case: $v=0$}.\\
Formula (2) is valid for all $0<u<1$ and $v=0$ under conditions $a),\ b)$ and $c)$. Then,  formula (3) can be rewritten in the form
\begin{eqnarray}
\sum_{\rho}\int_{-\infty}^{+\infty}&f(\frac{x}{u})&e^{x(\rho-\frac{1}{2})}dx\nonumber\\&=&-\sum_{n=2}^{+\infty}\frac{\Lambda_{F}(n)}{\sqrt{n}}f(\frac{\log n}{u})-\sum_{n=2}^{+\infty}\frac{\overline{\Lambda_{F}(n)}}{\sqrt{n}}f(\frac{-\log n}{u})\\
&&+m_{F}\left(\int_{-\infty}^{+\infty}f(\frac{x}{u})e^{\frac{x}{2}}dx+\int_{-\infty}^{+\infty}f(\frac{x}{u})e^{-\frac{x}{2}}dx\right)+2f(0)\log Q\nonumber\\
&&+\sum_{j=1}^{r}\lambda_{j}\left(\frac{\Gamma'}{\Gamma}(\frac{\lambda_{j}}{2}+\mu_{j})+\frac{\Gamma'}{\Gamma}(\frac{\lambda_{j}}{2}+\overline{\mu_{j}})\right)f(0)\nonumber\\
&&-\sum_{j=1}^{r}\lambda_{j}\int_{0}^{+\infty}\left( f\left(\frac{-\lambda_{j}x}{u}\right)-f(0^{-})\right)\frac{e^{-(\frac{\lambda_{j}}{2}+\mu_{j})x}}{1-e^{-x}}dx\nonumber\\
&&-\sum_{j=1}^{r}\lambda_{j}\int_{0}^{+\infty}\left( f\left(\frac{\lambda_{j}x}{u}\right)-f(0^{+})\right)\frac{e^{-(\frac{\lambda_{j}}{2}+\overline{\mu_{j}})x}}{1-e^{-x}}dx\nonumber
 \end{eqnarray}
Denote the right-hand side of (31) by $M_{1}+M_{2}+m_{F}M_{3}+m_{F}M_{4}+M_{5}+M_{6}+M_{7}$. By (23) and the same arguments of the first case, we have
\begin{equation}
M_{1},\ M_{2}=O(u),\ \ 0<u<1/2,
\end{equation}
and
\begin{equation}
 M_{3},\ M_{4}=O(u),\ \ 0<u<1.
\end{equation}
Now, let  treat the term $M_{6}$. We have
\begin{eqnarray}
-\int_{0}^{+\infty}&&\left( f\left(\frac{-\lambda x}{u}\right)-f(0^{-})\right)\frac{e^{-(\frac{\lambda}{2}+\mu)x}}{1-e^{-x}}dx\nonumber\\&=&\ -\int_{0}^{+\infty}\left( f\left(\frac{-\lambda x}{u}\right)-f(0^{-})\right)e^{-(\frac{\lambda}{2}+\mu)x}\frac{dx}{x}\\
&& +f(0^{-})\int_{0}^{+\infty}\left(\frac{1}{e^{x}-1}-\frac{1}{x}+1\right)e^{-(\frac{\lambda}{2}+\mu)x}dx\nonumber\\
 &&-\ \int_{0}^{+\infty}\left(\frac{1}{e^{x}-1}-\frac{1}{x}+1\right)e^{-(\frac{\lambda}{2}+\mu)x}f\left(\frac{-\lambda x}{u}\right)dx.\nonumber
\end{eqnarray}
Let
$$H(x)=\left(\frac{1}{e^{x}-1}-\frac{1}{x}+1\right)e^{-(\frac{\lambda}{2}+\mu)x}.$$
Then, the quantity given by (34) is equal to
\begin{eqnarray}
-\int_{0}^{+\infty}\left( f\left(\frac{-\lambda x}{u}\right)-f(0^{-})\right)e^{-(\frac{\lambda}{2}+\mu)x}\frac{dx}{x}&+&f(0^{-})\int_{0}^{+\infty}H(x)dx\nonumber\\&-&\int_{0}^{+\infty}H(x)f\left(\frac{-\lambda x}{u}\right)dx.
\end{eqnarray}
Because $H(x)$ is bounded on $(0,\infty)$, from (23) and integration by part,
$$\int_{0}^{+\infty}H(x)f\left(\frac{-\lambda x}{u}\right)=O(u).$$
By observing that
$$\int_{0}^{+\infty}\left(\frac{1}{e^{x}-1}-\frac{1}{x}+1\right)e^{-(\frac{\lambda}{2}+\mu)x}dx=\log\left(\frac{\lambda}{2}+\mu\right)-\frac{\Gamma'}{\Gamma}\left(\frac{\lambda}{2}+\mu\right),$$
the second term of (35) is
\begin{equation}
f(0^{-})\log\left(\frac{\lambda}{2}+\mu\right)-f(0^{-})\frac{\Gamma'}{\Gamma}\left(\frac{\lambda}{2}+\mu\right).
\end{equation}
Hence
\begin{eqnarray}
-\int_{0}^{+\infty}&&\left( f\left(\frac{-\lambda x}{u}\right)-f(0^{-})\right)\frac{e^{-(\frac{\lambda}{2}+\mu)x}}{1-e^{-x}}dx\nonumber\\&=&\ -\int_{0}^{+\infty}\left( f\left(\frac{-\lambda x}{u}\right)-f(0^{-})\right)e^{-(\frac{\lambda}{2}+\mu)x}\frac{dx}{x}\nonumber\\&&\ \ \  +\ f(0^{-})\log\left(\frac{\lambda}{2}+\mu\right)-f(0^{-})\frac{\Gamma'}{\Gamma}\left(\frac{\lambda}{2}+\mu\right)+O(u).
\end{eqnarray}
Using the fact (see. \cite{5}),
$$\int_{0}^{+\infty}\left(e^{-x}-e^{-\omega x}\right)\frac{dx}{x}=\log \omega,$$
we get
\begin{eqnarray}
&-&\int_{0}^{+\infty}\left( f\left(\frac{-\lambda x}{u}\right)-f(0^{-})\right)e^{-(\frac{\lambda}{2}+\mu)x}\frac{dx}{x}+f(0^{-})\log\left(\frac{\lambda}{2}+\mu\right)\nonumber\\
&=&-\int_{0}^{+\infty}\left( f\left(\frac{-\lambda x}{u}\right)-f(0^{-})\right)e^{-(\frac{\lambda}{2}+\mu)x}\frac{dx}{x}\nonumber\\&&+f(0^{-})\int_{0}^{+\infty}\left(e^{-x}-e^{-\left(\frac{\lambda}{2}+\mu\right) x}\right)\frac{dx}{x}\nonumber\\
&=&-\int_{0}^{+\infty}\left( f\left(\frac{-\lambda x}{u}\right)e^{-\left(\frac{\lambda}{2}+\mu\right) x}-f(0^{-})e^{-x}\right)\frac{dx}{x}\nonumber\\
&=&-\int_{0}^{+\infty}\left( f(-\lambda x)e^{-\left(\frac{\lambda}{2}+\mu\right) ux}-f(0^{-})e^{-ux}\right)\frac{dx}{x}\nonumber\\
&=&f(0^{-})\int_{0}^{+\infty}\left(e^{-\left(\frac{\lambda}{2}+\mu\right) ux}-e^{-ux}\right)\frac{dx}{x}-\int_{0}^{+\infty}\left(f(-\lambda x)-f(0^{-})e^{-(\frac{\lambda}{2}+\mu)x}\right)\frac{dx}{x}\nonumber\\
&&+ \int_{0}^{+\infty}\left(1-e^{-\left(\frac{\lambda}{2}+\mu\right)ux}\right)f(-\lambda x)\frac{dx}{x}.
\end{eqnarray}
The first term of (38) is
$$-f(0^{-})\log\left[\left(\frac{\lambda}{2}+\mu\right)u\right],$$ while
the third term is $ O(u)$. Therefore
\begin{eqnarray}
-\int_{0}^{+\infty}&&\left( f\left(\frac{-\lambda x}{u}\right)-f(0^{-})\right)\frac{e^{-(\frac{\lambda}{2}+\mu)x}}{1-e^{-x}}dx\nonumber\\&&=\ f(0^{-})\left[\log\left(\left(\frac{\lambda}{2}+\mu\right)u\right)-\frac{\Gamma'}{\Gamma}\left(\frac{\lambda}{2}+\mu\right)\right]\nonumber\\
&&\ -\ \int_{0}^{+\infty}\left(f(-\lambda x)-f(0)e^{-(\frac{\lambda}{2}+\mu)x}\right)\frac{dx}{x}.
\end{eqnarray}
The same argument used to prove (39) gives
\begin{eqnarray}
-\int_{0}^{+\infty}&&\left( f\left(\frac{\lambda x}{u}\right)-f(0^{+})\right)\frac{e^{-(\frac{\lambda}{2}+\overline{\mu})x}}{1-e^{-x}}dx\nonumber\\
&=&\ f(0^{+})\left[\log\left(\left(\frac{\lambda}{2}+\overline{\mu}\right)u\right)-\frac{\Gamma'}{\Gamma}\left(\frac{\lambda}{2}+\overline{\mu}\right)\right]\nonumber\\
&&\ -\ \int_{0}^{+\infty}\left(f(\lambda x)-f(0^{+})e^{-(\frac{\lambda}{2}+\overline{\mu})x}\right)\frac{dx}{x}.
\end{eqnarray}
This ends the proof of Theorem \ref{thm:2}.
\end{section}
\section{Proof of Theorem \ref{thm:3}}
Assume  that the Riemann hypothesis holds, that is, all non-trivial zeros have the form $\rho=1/2+i\gamma$. Let $f$ be a function as above, then
$$\widehat{f}(t)=O\left(\frac{1}{|t|}\right).$$
Let $u$ be such that $0<u<1$ and denote the $n^{\hbox{th}}$ value of $\gamma>0$ by $\gamma_{n}$. Bombieri and Hejhal in \cite{1} proved that \\
\begin{equation}N_{F}(T)=\frac{d_{F}}{2\pi}T\log T+c_{1}T+O(\log T),\end{equation}
where
$$c_{1}=\frac{1}{2\pi}\left(\log q_{F}-d_{F}(\log2\pi+1)\right)\ \hbox{and}\   q_{F}=(2\pi)^{d_{F}}Q^{2}\prod_{j=1}^{r}\lambda_{j}^{2\lambda_{j}}.$$
Using the same argument as in Guinand \cite[page 108]{4}, we can prove that
\begin{eqnarray}
&-&\int_{0}^{\gamma_{N+1}}\left(N_{F}(T)-\frac{d_{F}}{2\pi}T\log\left(\frac{T}{2\pi }\right)-\frac{1}{2\pi}T\left(\log q_{F}-d_{F}\right)\right)\frac{d}{dT}\left(\widehat{f}(-u T)\right)dT\nonumber\\
&&=\ \sum_{r=1}^{N}\widehat{f}(-u\gamma_{r})-\frac{d_{F}}{2\pi}\int_{0}^{\gamma_{N+1}}\widehat{f}(-uT)\log\left(\frac{T}{2\pi}\right)dT+
O\left(\frac{\log\gamma_{N+1}}{\gamma_{N+1}}\right).\nonumber
\end{eqnarray}
Consequently,
$$\sum_{r=1}^{+\infty}\widehat{f}(-u\gamma_{r})=\frac{1}{2u}\sum_{\gamma}\int_{-\infty}^{+\infty}f\left(\frac{x}{u}\right)e^{i\gamma x}dx,$$
which is finite. Therefore
\begin{eqnarray}
&-&\int_{0}^{\infty}\left(N_{F}(T)-\frac{d_{F}}{2\pi}T\log\left(\frac{T}{2\pi }\right)-\frac{1}{2\pi}T\left(\log q_{F}-d_{F}\right)\right)\frac{d}{dT}\left(\widehat{f}(-u T)\right)dT\nonumber\\
&&=\ \frac{1}{2u}\sum_{\gamma}\int_{-\infty}^{+\infty}f\left(\frac{x}{u}\right)e^{i\gamma x}dx-\frac{d_{F}}{2\pi}\int_{0}^{+\infty}\widehat{f}(-u T)\log\left(\frac{T}{2\pi }\right)dT\nonumber\\
&&\ +\  O\left(\frac{\log\gamma_{N+1}}{\gamma_{N+1}}\right).\nonumber
\end{eqnarray}
With the change of variable $T=\frac{t}{u\lambda}$ and by the Fourier inversion formula, we obtain
\begin{eqnarray}
-\frac{d_{F}}{2\pi}\int_{0}^{+\infty}&\widehat{f}(-uT)&\log\left(\frac{T}{2\pi }\right)dT\nonumber\\
&=&-\frac{d_{F}}{2\pi u\lambda}\int_{0}^{+\infty}\widehat{f}\left(-\frac{t}{\lambda}\right)\log\left(\frac{t}{2\pi u\lambda }\right)dt\nonumber\\
&=&-\frac{d_{F}}{2\pi u\lambda}\int_{0}^{+\infty}\widehat{f}\left(- \frac{t}{\lambda}\right)\left(\log t-\log(2\pi u\lambda )\right)dt\nonumber\\
&=&\frac{d_{F}\log(2\pi  u \lambda)}{2 u}f(0)-\frac{d_{F}}{2\pi u\lambda}\int_{0}^{+\infty}\widehat{f}\left(-  \frac{t}{\lambda}\right)\log tdt.\nonumber
\end{eqnarray}
 Hence
\begin{eqnarray}
&&\sum_{\gamma}\int_{-\infty}^{+\infty}f\left(\frac{x}{u}\right)e^{i\gamma x}dx\nonumber\\
&&\ \ \ \ \ \ \ \ =\ -d_{F}\log(2\pi  u\lambda)f(0)+\frac{d_{F}}{\pi \lambda}\int_{0}^{+\infty}\widehat{f}(- \frac{T}{\lambda})\log TdT\\
&&-2u\int_{0}^{\infty}\left(N_{F}(T)-\frac{d_{F}}{2\pi}T\log\left(\frac{T}{2\pi }\right)-\frac{1}{2\pi}T\left(\log q_{F}-d_{F}\right)\right)\frac{d}{dT}\left(\widehat{f}(-u T)\right)dT.\nonumber
\end{eqnarray}
The second term in the right-hand-side of the last equation (42) is given by  Lemma \ref{lem:1}.
 Then, Theorem \ref{thm:3} follows.
 \begin{section}{Appendix:  the Weil explicit formulas and the Li coefficients}
 In 1997, Xian-Jin Li has discovered a new positivity criterion for the Riemann hypothesis. In \cite{9}
he proved that the Riemann hypothesis is equivalent with the non-negativity of numbers
$$\lambda_{n}=\sum_{\rho}\left(1-\left(1-\frac{1}{\rho}\right)^{n}\right)$$
for all $n\in{\nb}$, where the sum is taken over all non-trivial zeros of the Riemann zeta function. A little later,
Bombieri and Lagarias \cite{2} observed that the Li criterion can be generalized to a multi-set
of complex numbers satisfying certain conditions, and gave an arithmetic formula for numbers $\lambda_{n}$. In \cite{14} and \cite{15}, it was shown that one could formulate a Li-type criterion for a general class of Dirichlet series, which includes elements of the Selberg class and obtain an arithmetic formula for the generalized Li coefficient defined below.

 In this appendix, we give another form of the Weil explicit formulas and use it to find an arithmetic formula for the generalized Li coefficients.\\

 By $\varphi BV$ we denote the set of functions of bounded $\varphi$-variation in the sense of L. C. Young.
A function $f$ is said to be of $\varphi$ bounded variation on an interval $I$ with the end points $a$ and $b$ if
$$V_{\varphi}(f, I) = sup\sum_{n}\varphi(|f(I_{n})|)<\infty,$$
where $f (I)$ stands for $f (b)-f (a)$ and the supremum is taken over all systems $\{I_{n}\}$ of non overlapping
subintervals of $I$.
 \begin{proposition} \label{prop:2} Let a regularized function $G$ fulfill the following conditions:
\begin{enumerate}
	\item $G\in{\varphi BV(\rb)\bigcap L^{1}(\rb)}$.
	\item $G(x)e^{(1/2+\epsilon)|x|}\in{\varphi BV(\rb)\bigcap L^{1}(\rb)}$, for some $\epsilon>0$.
	\item $G(x)+G(-x)-2G(0)=O\left(|\log |x||^{\alpha}\right)$, as $x\rightarrow0$, for some $\alpha>2$.
\end{enumerate}
Let $F(s)\in{{\mathcal S}}$. Then,
\begin{eqnarray}
\sum_{\rho}&&\widetilde{g}_{1/2}(\rho)\nonumber\\&=&m_{F}(\widetilde{g}_{1/2}(0)+\widetilde{g}_{1/2}(1))\nonumber\\
&&\ -\ \sum_{n}\frac{b_{F}(n)}{n^{1/2}}g(n)-\sum_{n}\frac{\overline{b_{F}}(n)}{n^{1/2}}g(1/n)+2G(0)\log Q_{F}\nonumber\\
&&\ +\ \sum_{j=1}^{r}\int_{0}^{+\infty}\left\{\frac{2\lambda_{j}G_{j}(0)}{x}-\frac{e^{\left(\left(1-\frac{\lambda_{j}}{2}-\Re(\mu_{j})\right)\frac{x}{\lambda_{j}}\right)}}{1-e^{-\frac{x}{\lambda_{j}}}}(G_{j}(x)+G_{j}(-x))\right\}e^{-\frac{x}{\lambda_{j}}}dx,\nonumber
\end{eqnarray}
where $ \rho$ runs over all non-trivial zeros  of $F(s)$ counted with multiplicity and $\widetilde{g}_{1/2}$ denotes the translate by $1/2$ of the Mellin transform of the function $g$.
 \end{proposition}
  Let $F$ be a function in the Selberg class non-vanishing at $s=1$ and let us define the xi-function $\xi_{F}(s)$ by
$$\xi_{F}(s)=s^{m_{F}}(s-1)^{m_{F}}\phi_{F}(s).$$
 The function $\xi_{F}(s)$ satisfies the functional equation
$$\xi_{F}(s)=\omega\overline{\xi_{F}(1-\overline{s})}.$$
The function $\xi_{F}$ is an entire function of order 1. Therefore, by the Hadamard product,
it can be written as
$$
\xi_{F}(s)=\xi_{F}(0)\prod_{\rho}\left(1-\frac{s}{\rho}\right),
$$
where the product is over all zeros of $\xi_{F}(s)$ in the order given by $|\Im(\rho)|<T$ for $T\rightarrow\infty$. Let $\lambda_{F}(n)$, $n\in{\zb}$, be a sequence of numbers defined by a sum over the non-trivial zeros of $F(s)$ as
$$\lambda_{F}(n)=\sum_{\rho}\left[1-\left(1-\frac{1}{\rho}\right)^{n}\right],$$
where the sum over $\rho$ is
$$\sum_{\rho}=\lim_{T\mapsto\infty}\sum_{|\Im\rho|\leq T}.$$
These coefficients are expressible in terms of power-series coefficients of functions constructed from the $\xi_{F}$-function. For $n\leq-1$, the Li coefficients  $\lambda_{F}(n)$  correspond to the following Taylor expansion at the point $s=1$
$$\frac{d}{dz}\log\xi_{F}\left(\frac{1}{1-z}\right)=\sum_{n=0}^{+\infty}\lambda_{F}(-n-1)z^{n}$$
and for $n\geq1$, they correspond to the Taylor  expansion at $s=0$
$$\frac{d}{dz}\log\xi_{F}\left(\frac{-z}{1-z}\right)=\sum_{n=0}^{+\infty}\lambda_{F}(n+1)z^{n}.$$
Let ${\mathcal Z}$ the multi-set of zeros of $\xi_{F}(s)$ (counted with multiplicity). The multi-set ${\mathcal Z}$ is invariant under the map $\rho\longmapsto1-\overline{\rho}$. We have
$$1-\left(1-\frac{1}{\rho}\right)^{-n}=1-\left(\frac{\rho-1}{\rho}\right)^{-n}=1-\left(\frac{-\rho}{1-\rho}\right)^{n}=1-\overline{\left(1-\frac{1}{1-\overline{\rho}}\right)^{n}}$$
 and this gives the symmetry $\lambda_{F}(-n)=\overline{\lambda_{F}(n)}$. Using the corollary in \cite[Theorem 1]{2}, we get the following generalization  of the Li criterion for the Riemann hypothesis.
\begin{theorem} \label{thm:4}\cite{14} \cite{16} Let $F(s)$ be a function in the Selberg class ${\mathcal S}$ non-vanishing at $s=1$. Then, all non-trivial zeros of $F(s)$ lie in the line $\Re e(s)=1/2$ if and only if $\Re e\left(\lambda_{F}(n)\right)\geq0$ for $n=1,2..$.
\end{theorem}
Let consider the following hypothesis: \\
{\bf ${\mathcal H}$: there exists a constant $c>0$ such that $F(s)$ is non-vanishing in the region:
$$\left\{s=\sigma+it;\ \sigma\geq1-\frac{c}{\log(Q_{F}+1+|t|)}\right\}.$$}
Let Consider $g(x)=G(-\log x)$, for $x>0$ and $G_{j}(x)=G(x)e^{\left(\frac{ix \Im(\mu_{j})}{\lambda_{j}}\right)}$. Applying Proposition \ref{prop:2} to the function
 $$G_{n,z}(x)=\left\{\begin{array}{crll}e^{-(z+x/2)}\sum_{l=1}^{n}\left(_{l}^{n}\right)\frac{(-x)^{l-1}}{(l-1)!}&\hbox{if}&x>0,\\
 n/2&\hbox{if}&x=0,\\0&\hbox{if}&x<0,\end{array}\right.$$
 where $z$ is a positive constant. We obtain the following Theorem.
 \begin{theorem}\label{thm:5} \cite{14} \cite{15} Let $F(s)$ be a function in the Selberg class ${\mathcal S}$ satisfying ${\mathcal H}$. Then, we have
\begin{eqnarray}
\lambda_{F}(-n)&=&m_{F}+n(\log Q_{F}-\frac{d_{F}}{2}\gamma_{0})\nonumber\\
&-&\sum_{l=1}^{n}(_{l}^{n})\frac{(-1)^{l-1}}{(l-1)!}\ \lim_{X\longrightarrow+\infty}\left\{\sum_{k\leq X}\frac{\Lambda_{F}(k)}{k}(\log k)^{l-1}-\frac{m_{F}}{l}(\log X)^{l}\right\}\nonumber\\
&+&n\sum_{j=1}^{r}\lambda_{j}\left(-\frac{1}{\lambda_{j}+\mu_{j}}+\sum_{l=1}^{+\infty}\frac{\lambda_{j}+\mu_{j}}{l(l+\lambda_{j}+\mu_{j})}\right)\nonumber\\
&-&\sum_{j=1}^{r}\sum_{k=2}^{n}(_{k}^{n})(-\lambda_{j})^{k}\sum_{l=0}^{+\infty}\left(\frac{1}{l+\lambda_{j}+\mu_{j}}\right)^{k},
\end{eqnarray}
where $\gamma_{0}$ is the Euler constant.
 \end{theorem}
 In a recent works \cite{10} and \cite{17}, new asymptotic formula was given for the generalized Li coefficients. To do so, we used two  different methods, both of them give the same main term. The first  is inspired from  Lagarias method  yields to a sharper error term  $O(\sqrt{n}\log n)$,  while the second use the saddle-point method.
 \begin{theorem}\label{thm:6} \cite{10} \cite{17} Let $F\in{{\mathcal S}}$. Then
$$RH\Leftrightarrow\lambda_{F}(-n)=\frac{d_{F}}{2}n\log n+c_{F}n+O(\sqrt{n}\log n),$$
where $$c_{F}=\frac{d_{F}}{2}(\gamma_{0}-1)+\frac{1}{2}\log(\lambda Q_{F}^{2}),\ \ \lambda=\prod_{j=1}^{r}\lambda_{j}^{2\lambda_{j}}$$
and $\gamma_{0}$ is the Euler constant.
\end{theorem}
\end{section}
%
%
\bigskip
{\em Acknowledgements.}
We would like to express our thanks to the anonymous referee for his/her careful reading of the manuscript, comments and suggestions. I am also grateful to  Prof. Aleksandar Ivi\'c  and Prof. Maciej Radziejewski for their  suggestions  that increased the clarity of the presentation.


%
%
\renewcommand\refname{References}

%


\begin{thebibliography}{99}
%
%
%
\bibitem[1]{1} E. Bombieri, D. A. Hejhal, {\it On the distribution of zeros of linear combinations of Euler products.}  Duke. Math. J.  80, no.{\bf 3} (1995), 821--862.
\bibitem[2]{2} E. Bombieri and J. C. Lagarias, {\it Complements
to Li's criterion for the Riemann hypothesis,} J. Number theory {\bf
77} (2) (1999),  274--287.
\bibitem[3]{3} S. M. Gonek, {\it An explicit formula of Landau and its applications to thetheory of the zeta-function.} Knopp, Marvin (éd.) et al., A tribute to EmilGrosswald: number theory and related analysis. Providence, RI: AmericanMathematical Society. Contemp. Math. {\bf 143} (1993), 395--413.
\bibitem[4]{4} A. P. Guinand, {\it A summation formula in the theory of prime numbers.} Proc. London Math. Soc. (2) {\bf 50} (1950), 401--414.
\bibitem[5]{5} N. N. Lebedev, {\it Special Functions and their applications.} Dover Publications, 1972.
\bibitem[6]{6} J. Kaczorowski, A. Perelli, {\it On  the Selberg class : Survey.} Acta. Math. {\bf 192} (1999),  953--992.
\bibitem[7]{7} Y. Kamiya, {\it An attempt to interpret the Weil explicit formula from Beurling's spectral theory.} Journal of Number Theory
{\bf 131} (4) (2011), 685–-704.
\bibitem[8]{8} Y. Kamiya, M. Suzuki, {\it An asymptotic formula for a sum involving zeros of the Riemann zeta-function.} Publication de l'institut Mathématiques, Nouvelle série, tome {\bf 76} (90) (2004), 81--88.
\bibitem[9]{9} X.-J. Li, {\it The positivity of a sequence of numbers and the Riemann hypothesis.} J. Number theory {\bf 65} (2) (1997)  325--333.
\bibitem[10]{10}  K. Mazhouda, {\it The saddle-point method and the Li coefficients.}  Can. Math. Bull. {\bf54}, No. 2  (2011), 316--329.
\bibitem[11]{11} J. F. Mestre, {\it Formules explicites et minorations de conducteurs de variétées algébriques.} Compositio Mathematica. {\bf 58} (1986), 209-232.
\bibitem[12]{12} R. Murty, A. Perelli, {\it  The pair correlation of zeros of function in the Selberg class.} IMRN. No. {\bf 10} (1999),  531--545.
\bibitem[13]{13} R. Murty, A. Zaharescu, {\it Explicit formulas for the pair correlation of zeros of function in the  class.} Forum Math. {\bf 14} (2002), 65--83.
\bibitem[14]{14} S. Omar, K. Mazhouda, {\it  Le critère de Li et l'hypothèse de Riemann pour la classe de Selberg.} Journal of Number Theory {\bf 125} (1) (2007), 50--58.
\bibitem[15]{15} S. Omar, K. Mazhouda, {\it Corrigendum et addendum à l'article, "Le critère de Li et l'hypothèse de Riemann pour la classe de Selberg"  [J. Number Theory {\bf 125} (2007) 50--58]. } Journal of Number Theory {\bf 130} (4) (2010), 1109--1114.
\bibitem[16]{16} S. Omar, K. Mazhouda, {\it Le critère de positivité de  Li pour la classe de Selberg,}  C. R. Acad. Sci. Paris, Ser. I {\bf 345} (5) (2007), 245--248.
\bibitem[17]{17} S. Omar, K. Mazhouda, {\it Li's criterion and the Riemann hypothesis for the Selberg class II,} Journal of Number Theory {\bf 130} (4)   (2010), 1098--1108.
\bibitem[18]{18} A. Selberg, {\it Old and new conjectures and results about a class of Dirichlet series.} presented at the Amelfi conference on number theory, Septembre  1989.
\bibitem[19]{19} R. Slezeviciene, J. Steuding, {\it Short series over simple zeros of the Riemann zeta-function,} Indagationes math. {\bf15} (2004), 129-132.
\bibitem[20]{20} M. Suzuki, {\it A relation between the zeros of different two $L$-functions which have the Euler product and functional equation.} Arxiv.math.NT/0412224.
\end{thebibliography}
\end{document}